\documentclass{amsart}

\usepackage{amsmath}
\usepackage{amsthm}
\usepackage{amsfonts}
\usepackage{dsfont}
\usepackage{enumerate}
\usepackage{color}
\usepackage{amssymb}
\usepackage{pst-node}
\usepackage{tikz-cd} 
\allowdisplaybreaks[4]

\newcommand{\R}{\mathbb R}

\newcommand{\C}{\mathbb C}

\newcommand{\Z}{\mathbb Z}
\newcommand{\T}{\mathbb T}

\newtheorem{theorem}{Theorem}[section]

\newtheorem{lemma}[theorem]{Lemma}
\newtheorem{proposition}[theorem]{Proposition}

\theoremstyle{definition}
\newtheorem{definition}[theorem]{Definition}
\theoremstyle{remark}

\numberwithin{equation}{section}

\newcommand{\NCS}[1]{S^{#1}_{\theta'}}

\newcommand{\sut}{SU(3)_\theta}

\newcommand{\id}{\operatorname{id}}

\title[]{Connes-Landi spheres are homogeneous spaces}
\author{Mitsuru Wilson}

\address[Mitsuru Wilson]{%
Edificio H\\
Universidad de los Andes\\
Bogot\'a, Colombia\\
N6A 5B7\\
Canada}
\email{m.wilson@uniandes.edu.co}

\subjclass[2000]{}

\begin{document}

\begin{abstract}
In this paper, we review some recent developments of compact quantum groups that arise as $\theta$-deformations of 
compact Lie groups of rank at least two. A $\theta$-deformation is merely a 2-cocycle deformation using an action of a torus of dimension higher than 2. 
Using the formula (Lemma 5.3) developed in \cite{W2018}, we derive
the noncommutative 7-sphere in the sense of Connes and Landi \cite{CL2001} as the fixed-point subalgebra.
\end{abstract}

\maketitle

\tableofcontents

\noindent
Lie groups play crucial roles both in mathematics and physics. For example, most field theories of particle physics are based on certain symmetries with gauge group. 
For instance, quantum chromodynamics is a gauge theory with the symmetry group $SU(3)$ while quantum electromagnetism is a gauge theory with the symmetry 
group $U(1)$ and  the standard model is a gauge theory with the symmetry group $SU(3)\times SU(2)\times U(1)$. Also, $SU(5)$ had been proposed as a 
gauge group in grand unification theory. 

The main purpose of this article is to survey the construction of compact quantum groups from compact Lie groups of rank at least 2 using both the left and the right 
action of the torus on the algebra generated by coordinate functions \cite{MR1993.5}. Once we establish an action by $\T^n$, we can apply the 
$\theta$-deformation on the coordinate functions using an antisymmetric bicharacter or a 2-cocycle to deform the multiplication of 
the algebra \cite{CL2001, MR1993.5, MR1993}. These noncommutative manifolds are called $\theta$-deformation because the original work used 
$\theta$ as the parameter.

We consider a noncommutative version of the group $SU(n)$ in the framework of $\theta$-deformation of manifolds introduced 
in \cite{CL2001}. This version of deformations of compact Lie groups to compact quantum 
groups was first discovered by Rieffel in \cite{MR1993.5} using his theory of strict deformation quantization along the action of $\R^n$ \cite{MR1993}, 
which is a 2-cocycle twist of the multiplication of the algebra structure. Later, Connes and Landi in \cite{CL2001} considered isometric actions of $\T^n$ 
on compact spin manifolds to deform the algebra of continuous functions to obtain a noncommutative compact spin manifold. Since in their case, the spectrum 
of the Dirac operator does not change, it is often referred to as the isospectral deformation. More generally, since the new product on the algebra can be defined 
as long as there is an action of the $n$-torus on the algebra, the deformed algebra is also called a toric noncommutative manifold. Although this type of 
deformation was done in the C$^*$-algebraic framework, we will deform only the algebraic part to obtain generators and relations. At the algebraic level we will be 
able to give very explicit formulae in terms of the generators. The resulting compact quantum group will be a compact matrix quantum 
group in the sense of Woronowicz \cite{W1987}.

In \cite{CD2002}, Connes and Dubois-Violette constructed some quantum matrix groups as quotients of the 
$\theta$-deformation $M_{n\times n}(\R)_\theta$ of the algebra generated by the coordinates of $M_{n\times n}(\R)$. 
However, their method does not yield any nontrivial deformation of $SU(n)$ whereas \cite{MR1993.5} yields nontrivial deformations of $SU(n)$, 
$n\geq3$. Since a compact Lie group $G$ of rank higher than 2 admits a nontrivial $\theta$-deformation to a compact quantum group, it is natural to ask 
whether a Lie group action $G\times X\rightarrow X$ on a compact manifold $X$ endowed with a $\T^n$ action, $n\geq2$, 
extends to an action $C(X_{\theta'})\rightarrow C(G_\theta)\otimes C(X_{\theta'})$ of the compact quantum group $G_\theta$. 
Of course, the original action will induce an action at the
vector space level, but, in general, it is no longer an algebra homomorphism. For instance, Landi and van Suijlekom in \cite{LvS2005} studied for exactly which 
values $\theta_{ij}'$ of $S^7_{\theta'}$ the diagonal action of the group $SU(2)$ on $S^7$ extends to an action of $SU(2)$ on $S^7_{\theta'}$. 
On the contrary, the group $SU(2)$ does not admit any nontrivial $\theta$-deformation because the group is of rank 1. 
We recite a lemma from \cite{W2018} in order to determine the dependence of the deformation parameters $\theta$ and $\theta'$, which generalizes the computation 
in \cite{LvS2005}.  Although the group $SU(2)$ does not admit a nontrivial deformation in the framework of \cite{MR1993.5}, a similar construction of Hopf fibration was 
generalized to the quantum group $SU_q(2)$ in \cite{LPR2006}, which is, a priori, a different deformation than our framework.

An interesting consideration of $G_\theta$ is the construction of quantum homogeneous space. Suppose $\T^n\subset K\subset G$, $n\geq2$ with $K$ a 
closed subgroup of the compact Lie group $G$. Using the isomorphism $C(G/K)\cong C(G)^K$, we can define a quantum homogeneous space to be the fixed-point 
subalgebra for the action of $K$. This definition can be generalized to essentially any quantum subgroup acting on the quantum group. 
In our present work, we present an example of the 'quotient' $SU(4)_{\lambda}/SU(3)_\theta$. To make this statement precise, we give an action 
$\rho:SU(4)_{\lambda}\rightarrow SU(4)_{\lambda}\otimes SU(3)_{\theta}$ and compute the invariant elements. 
We show that the subalgebra generated by the invariant elements is isomorphic to a noncommutative 7-sphere $S^7_{\theta'}$ for some $\theta'$.

Quantum homogeneous spaces using the $\theta$-deformation of compact Lie groups had been treated by Varilly in \cite{V2001}. However, Varilly 
does not consider an action of quantum groups nor the fixed-point subalgebra for the action. Rather than computing the fixed-point subalgebra, Varilly endows $C(G/K)$ 
with a new product consistent in a way that it is embedded in $C(G_\theta)$. We take an approach to endow $C(G)$ with the action (by the left 
multiplication) $C(G)\rightarrow C(K)\otimes C(G)$ by its subgroup $K$ then extend the action to the $\theta$-deformation. Only then, we can compute the fixed-point 
subalgebra for the action. This subalgebra will be our notion of quantum homogeneous spaces.

This paper is organized as follows. 
In Section \ref{sec:deformation}, we relate the strict deformation quantization of periodic action in \cite{MR1993} and the $\theta$-deformation in \cite{CD2002}. 
Section \ref{subsec:quantum.group} is devoted to the review of the construction of the $\theta$-deformation of compact quantum groups \cite{MR1993.5}, and we survey 
some main results. In Section \ref{subsec:SU(3)}, we compute relations on the generators of the deformation $SU(3)_\theta$ of $SU(3)$ as a compact quantum matrix group. 
In Section \ref{sec:action.S5}, we restate the necessary and the sufficient condition as to when an action of a group on an algebra of functions extends to the setting of 
$\theta$-deformations \cite[Lemma 5.3]{W2018} as Lemma \ref{lemma:action.condition}. 
We recall an example from \cite{W2018} in 
Section \ref{subsec:action.S5} where it endows the noncommutative 5-sphere $S^5_{\lambda}$ with an action of $SU(3)_\theta$. 
In this case, the fixed-point subalgebra is trivial. 
Finally in Section \ref{subsec:homogeneous.space}, we construct an action 
$\rho:SU(4)_{\lambda}\rightarrow SU(4)_{\lambda}\otimes SU(3)_{\theta}$ of $SU(3)_\theta$ on $SU(4)_{\lambda}$ whose fixed-point subalgebra
$SU(4)_{\lambda}^{SU(3)_\theta}$ is isomorphic to the noncommutative 7-sphere $S^7_{\theta'}$.

\section{\bf The $\theta$-deformation of compact quantum groups}
\label{sec:deformation}

In this section, we review the construction of toric noncommutative manifolds as a special 
case of Rieffel's deformation quantization along the action of $\R^n$ \cite{MR1993} and use it to obtain a deformation of compact Lie group of rank at least 2 \cite{MR1993.5}
in Section \ref{subsec:quantum.group}. 
Deforming the algebraic structure is a simple matter while obtaining a compatible coalgebra structure is more subtle; the 2-cocycle has to be chosen carefully.

\subsection{\bf $\theta$-deformation as Rieffel's strict deformation quantization}
\label{subsec:rieffel.deformation}
\noindent 
The construction of toric noncommutative manifolds proceeds as follows \cite{CL2001, MR1993}.
Given a compact manifold $M$ endowed with an action of an $n$-torus $\T^n$ or equivalently with a periodic action of $\R^n$, 
$n\geq2$, the algebra $C^\infty(M)$ of smooth functions on $M$ can be decomposed into 
isotypic components $C^\infty(M)_{\vec r}=\big\{f\in C^\infty(M):\alpha_t(f)=e^{2\pi it\cdot\vec r}f\big\}$ indexed by $\vec r\in\widehat{\T^n}=\Z^n$. 
A deformed algebra structure can, then, be given by the linear extension of the product of two functions $f_{\vec r}$ and $g_{\vec s}$ in some 
isotypic components. A new product $\times_\theta$ on these elements is given by
\begin{equation}
f_r \times g_s = \chi(r,s)f_r g_s,
\label{eq:new.product}
\end{equation}
where $\chi: \Z^n \times \Z^n \rightarrow \T$ is an antisymmetric bicharacter on the Pontryagin dual 
$\Z^n=\widehat\T^n$ of $\T^n$ i.e. $\overline{\chi(s,r)}=\chi(r,s)$. The bicharacter relation
$$
\chi(r,s+t)=\chi(r,s)\chi(r,t),\qquad\chi(r+s,t)=\chi(r,t)\chi(s,t), \qquad s,t\in\Z^n
$$
ensures the associativity of the new product. 
For instance, 
$$
\chi(r,s) := \exp\left(\pi i~r\cdot\theta s \right)
$$
where $\theta = \left(\theta_{jk}\right)$ is a real antisymmetric $n \times n$ matrix is a typical choice. 
The involution for the new product is given by the complex conjugation 
of the functions. We denote the algebra obtained by extending the new product $\times_\theta$ to all continuous functions by $C(M_\theta)$ \cite{MR1993}.

The definition of a new product in \eqref{eq:new.product} is a discretized version of the Rieffel's deformation quantization \cite{MR1993} 
by viewing the action as a periodic action of $\R^n$. First of all, Rieffel expressed the deformed product of an algebra endowed with an action 
of $\R^n$ in the integral form
\begin{align}
\label{eq:deform.product}
a \times_J b
= \int_{V\times V} \alpha_{Js}(a) \alpha_t(b) ~e^{2\pi is\cdot t} ~ds~dt
\end{align}
where $a$ and $b$ belong to an algebra $A$ and $J$ is a real antisymmetric real $n \times n$ matrix \cite{MR1993}. This integral may be interpreted 
as an oscillatory integral and $\alpha: V \hookrightarrow$Aut$(A)$ is a strongly continuous action of a finite dimensional vector space $V \cong \R^n$ on 
$A$. The oscillatory integral \eqref{eq:deform.product} makes sense a priori only for elements $a,b$ of the smooth subalgebra $A^\infty$ (which forms 
a Fr\'echet algebra) of $A$ under the action $\alpha$ i.e. $v\mapsto\alpha_v(a)$ is smooth. However, Rieffel showed that if $A$ is a $C^*$-algebra, 
then the new algebra $A_J$ can also be given a C$^*$-algebra structure.

First, $A^\infty$ can be given a suitable pre $C^*$-norm for which the product $\times_J$ is 
continuous in a way that the completion with respect to this norm obtains the deformed algebra $A_J$. Thus, the deformed product can be extended 
to the entire algebra $A_J$.

It should be remarked that the smooth subalgebra remains unchanged as vector spaces $(A_J)^\infty = A^\infty$, even though they have different 
products \cite[Theorem 7.1]{MR1993.5}.

If the action $\alpha$ of $V$ is periodic, then $\alpha_v(a)= a$ for all $v$ in some lattice $\Lambda\subset V$ and $a\in A$. $\alpha$ can, then, be 
viewed as an action of the compact abelian group $H = V/\Lambda$. Then, $A^\infty$ admits a decomposition into a direct sum of isotypic components 
indexed by $\Lambda=\widehat H$. It is shown in \cite[Proposition.~2.21]{MR1993} that if $\alpha_s(a_p) = e^{2\pi ip\cdot s} a_p$ and $\alpha_t(b_q) = 
e^{2\pi iq\cdot t} b_q$ with $p,q \in\Lambda$, then \eqref{eq:new.product} defines an associative product.

The toric noncommutative manifolds \cite{CL2001} based on the deformation quantization \eqref{eq:new.product}, as far as the algebra structure is concerned, 
is a special case of Rieffel's strict deformation quantization theory. The same remarks are also made in \cite{Si2001} and \cite{V2001}. We call the $\theta$-deformed 
algebra the toric noncommutative manifold.

\subsection{\bf Compact quantum Lie groups associated with $n$-torus}
\label{subsec:quantum.group}

\noindent
We now review the $\theta$-deformation of compact Lie groups  in \cite{MR1993.5}. 

Let $G$ be a compact Lie group of rank at least $2$ or 
equivalently $\T^n\subset G$, $n\geq2$. We use the natural action $\alpha_{(t,s)}(f)(x)=f\left(t^{-1}xs\right)$ 
of $\T^n\times\T^n$ to deform the algebra $C(G)$ of continuous functions and write down the relations for the matrix coefficients in the case of $G=SU(3)$ and 
$SU(4)$. However, not every choice of deformation of this kind respects the original coalgebra structure \cite{MR1993}. If the antisymmetric matrix $\theta$ 
is chosen to be of the form $K\oplus(-K)$ where $K^T=-K$ is an $n\times n$ real antisymmetric matrix, we would obtain a $\theta$-deformation of compact 
quantum groups.

\begin{definition}
A coalgebra $C$ is an associative unital algebra with linear maps $\triangle:C\rightarrow C\otimes C$ called coproduct and $\epsilon:C\rightarrow\C$ called 
counit such that 
\begin{align*}
(\id\otimes\triangle)\circ\triangle&=(\triangle\otimes\id)\circ\triangle\\
(\epsilon\otimes\triangle)\circ\triangle&=(\triangle\otimes\epsilon)\circ\triangle=\id.
\end{align*}
If these maps are algebra homomorphisms, then $C$ is called a bialgebra. %Furthermore, if $C$ is a C$^*$ algebra we require that these maps are continuous.

If $H$ is a unital bialgebra endowed with an anti-homomorphism $S:H\rightarrow H$ such that 
\begin{align*}
m\circ(\id\otimes S)\circ\triangle=m\circ(S\otimes\id)\circ\triangle=\eta\circ\epsilon
\end{align*}
where $m:H\otimes H\rightarrow H$ is the multiplication map and $\eta:\C\rightarrow H$ is the embedding of the unit element
is called a compact quantum group.
\end{definition}

For example, if $G$ is a compact Lie group, then the usual (C$^*$-)compact quantum group structure can be given as follows:
\begin{align*}
&\triangle:C(G)\rightarrow C(G)\hat\otimes C(G)\cong C(G\times G),\qquad \triangle(f)(x,y)=f(xy)\\
&\epsilon(f)=f(e)\\
&S(f(x))=f(x^{-1})
\end{align*}
where $e\in G$ is the identity element and $\hat\otimes$ is the completed tensor product. Although in this paper we do not address the C$^*$-algebraic setting, 
when looking at the algebra of continuous functions on a compact Lie group, it is the correct setting.

With the notations above, let $K^T= - K$ be a real antisymmetric matrix. It was shown in \cite{MR1993.5} that the formula \eqref{eq:new.product} 
with $\theta=K\oplus(-K)$ with the unaltered coalgebra structure endows the algebra of continuous functions on the 
compact Lie group with a quantum group structure. Other choices of $\theta$ will not afford a quantum group with the original coalgebra structure.
We denote the resulting compact quantum group by $G_\theta$.  We view $\theta_{ij}\in\R$ as parameters of the deformation
Moreover, the deformation $(G_\theta)^\infty$ of smooth functions on $G$ remains dense in $C(G_\theta)$ \cite{MR1993.5}. 
In fact, if $G$ is a compact matrix group, then so too is $G_\theta$ a compact matrix quantum group in the sense of Woronowicz \cite{W1987}.
We will restrict our deformation to the coordinate functions of a compact matrix Lie group in order to be able to compute the commutation relations on them. 

%In this paper we treat $C(G_\theta)$ as a C$^*$-algebra. 
It is shown in \cite{MR1993.5} that the coproduct defined by $\triangle(f)(x,x')=f(xx')$ extends to a continuous homomorphism 
$\triangle:C(G_\theta)\rightarrow C(G_\theta)\otimes C(G_\theta)$. Here, the interpretation of $C(G_\theta)\otimes C(G_\theta)$ 
is given by the isomorphism $C\left((G\times G)_{\theta\oplus\theta}\right)\cong C(G_\theta)\otimes C(G_\theta)$ where the tensor product is the 
minimal tensor product \cite{MR1993} (in fact, $C(G_\theta)$ would be nuclear so distinguishing the kind of tensor product is not necessary in this setting but 
since this deformation applies to any C$^*$ quantum group endowed with an action of $\T^n$, we stress that this analysis is valid for the minimal tensor for 
a future consideration). Moreover, the counit $\epsilon(f)(x)=f(e)$ remain a homomorphism and the coinverse $S(f)(x)=f(x^{-1})$ 
remain an anti-homomorphism, and they satisfy all the compatibility conditions in the deformation. 

Lastly, since the normalized Haar measure determines a linear functional $\mu$ on $C(G)$, $\mu$ becomes a state on $C(G_\theta)$. In fact, this state is a 
tracial state $\mu_\theta:C(G_\theta)\rightarrow\C$ defined by $\mu_\theta(f):=\mu(f)=\int_Gf(x)dx$. Moreover, this tracial state is a Haar state 
\cite[Theorem 4.2]{MR1993.5} as we restate in the following:
\begin{theorem}\label{theorem:Haar.state}
The Haar measure $\mu$ on $C(G)$ determines a Haar state on the quantum group $C(G_\theta)$. 
That is, a continuous linear functional $\mu_\theta$ that is unimodular in the sense that
\begin{align}
\begin{split}
\left(id\otimes\mu_\theta\right)\circ\Delta&=\iota\circ\mu_\theta,\\
\left(\mu_\theta\otimes id\right)\circ\Delta&=\iota\circ\mu_\theta,
\end{split}\\
\begin{split}
\left(id\otimes\mu_\theta\right)\left(1\otimes a\right)\left(\Delta b\right)&=
\left(id\otimes\mu_\theta\right)\left(\left(S\otimes id\right)\Delta a\right)\left(1\otimes b\right),
\end{split}\\
\begin{split}
\mu_\theta\circ S&=\mu_\theta.
\end{split}
\end{align}
\end{theorem}
We denote this faithful trace simply by $\mu$ when the presence of $\theta$ is understood. 

In \cite{MR1993.5}, Rieffel determined the unitary dual of $G_\theta$. He showed that irreducible 
representations of $G$ are irreducible representations of $G_\theta$ and vice versa. Here, a representation $\pi$ of the compact Lie group $G$ 
means a continuous group homomorphism 
$\pi:G\rightarrow GL(V)$ for some finite dimensional complex vector space such that $\pi(g)\pi(g)^*=I_{\dim V}=\pi(g)^*\pi(g)$. $\pi$ is called irreducible if there is 
no nontrivial proper subspace $\{0\}\neq W\subset V$ such that $\pi(g)\cdot W\subset W$. In the quantum group case, we dualize the action $(g,v)\mapsto\pi(g)v$. 
A (unitary) representation of the quantum group $G_\theta$ is a unitary element $U\in M_{n}(G_\theta)$, $UU^*=U^*U=I_n$ such that $(id\otimes\triangle)(U)=
U_{12}U_{13}$.

Unlike in the case of classical representation theory, 
if $\pi$ and $\rho$ are representations of $G$ on $V$ and on $W$, respectively, then $V\otimes W$ 
and $W\otimes V$ are equivalent, but the equivalence is not given by the flip map $\sigma_{\pi\rho}:V\otimes 
W\rightarrow W\otimes V$, $\sigma_{\pi\rho}(v\otimes w)=w\otimes v$. 
It is possible to decompose a $G$-representation $V$ into $\T^n$-representations where $\T^n\subset G$. Then, every $v\in V$ admits the decomposition 
$v=\sum_{\vec p\in\Z^n}v_{\vec p}$ where $v_{\vec p}=\int_{\T^n}\alpha_t(v)e^{-2\pi ip\cdot t}dt$.

Then, the equivalence $\pi\otimes\rho\sim\rho\otimes\pi$ is induced by the linear extension of
\begin{align}\label{eq:product.representation}
R_{\pi\rho}(v_{\vec p}\otimes w_{\vec q})= e^{\pi i \theta(\vec p,\vec q)}v_{\vec p}\otimes w_{\vec q}~.
\end{align}
This is not surprising since \cite[Proposition 2.4]{W1987} showed that a compact matrix group with the property that 
if the map $\sigma$ for every pair of representations is an interwining operator then the quantum group is necessarily 
commutative. Moreover, one can view the above map $R_{\pi\rho}$ as the corresponding $\theta$-deformation of the tensor product.

\subsection{\bf $\theta$-deformations of $SU(3)$}
\label{subsec:SU(3)}

\noindent In this section, we explicitly compute the $\theta$-deformation $SU(3)_\theta$ of the Lie group $SU(3)$ using
the approach in Section \ref{subsec:quantum.group}.

We use the maximal torus $\T^2\subset SU(3)$ to deform the algebra $H$ generated by the matrix coordinate functions in $C(SU(3))$.
The coordinate functions $u_{ij}$ of $U\in SU(3)$ satisfy the relation
\begin{align*}
&\sum_{k=1}^3 u_{jk}\overline{u_{jl}}=\delta_{kl},\qquad
\sum_{k=1}^3 \overline{u_{kj}}u_{lj}=\delta_{kl},\qquad \det U=1.
\end{align*}
Let $K=\begin{pmatrix}0 & \theta\\-\theta &0\end{pmatrix}$ and we use the diagonal representation of the maximal torus
$$\T^2=\left\{t=
\begin{pmatrix}
e^{2\pi i\varphi_1} & 0 & 0\\
0 & e^{2\pi i\varphi_2} & 0\\
0 & 0 & e^{2\pi i\varphi_3}
\end{pmatrix}:\varphi_j\in\R,~\varphi_3=-(\varphi_1+\varphi_2)\right\}.
$$
In fact, monomials in coordinates are the isotypic components of the action 
\begin{align*}
&\alpha_{(t,s)}(f(U))=f\left(tUs\right)\\
&=f\begin{pmatrix}
e^{2\pi i\left(\varphi_{1} + \psi_{1}\right)}u_{11} & e^{2\pi i\left(\varphi_{1} + \psi_{2}\right)}u_{12} & e^{2\pi i\left(\varphi_{1} + \psi_{3}\right)}u_{13}\\
e^{2\pi i\left(\varphi_{2} + \psi_{1}\right)}u_{21} & e^{2\pi i\left(\varphi_{2} + \psi_{2}\right)}u_{22} & e^{2\pi i\left(\varphi_{2} + \psi_{3}\right)_{2}}u_{23}\\
e^{2\pi i\left(\varphi_{3} + \psi_{1}\right)}u_{31} & e^{2\pi i\left(\varphi_{3} + \psi_{2}\right)}u_{32} & e^{2\pi i\left(\varphi_{3} + \psi_{3}\right)}u_{33}
\end{pmatrix}
\end{align*}
of $\T^2\times\T^2$. In the above, we omitted taking the inverse of $t$ because $\T^n$ is Abelian. Since $K$ is a 2$\times2$ real antisymmetric matrix, 
there is only one independent parameter present in the deformation, we denote the resulting compact quantum group by $SU(3)_\theta$.
It is evident that each coordinate function is in an isotypic component $A_{\vec n}$, $\vec n={(n_1,n_2,n_3,n_4)}$, of this action. Then using the formula 
\eqref{eq:new.product},
\begin{align*}
u_{ij}\times_\theta u_{kl}=e^{\pi i \theta(- n_1m_2 + n_2m_1 + n_3m_4 - n_4m_3)}u_{ij}u_{kl},
\end{align*}
$u_{ij}\in A_{\vec n}$ and $u_{kl}\in A_{\vec m}$.
Thus, we obtain the following commutation relations:
\begin{alignat*}{3}
&u_{11}u_{12}=e^{-2\pi i\theta}u_{12}u_{11},&\quad& u_{11}u_{13}=e^{2\pi i\theta}u_{13}u_{11},&\quad& u_{11}u_{21}=e^{2\pi i\theta}u_{21}u_{11},\\
&\left[u_{11},u_{22}\right]=0,&&u_{11}u_{23}=e^{4\pi i\theta}u_{23}u_{11},&& u_{11}u_{31}=e^{-2\pi i\theta}u_{31}u_{11},\\
&u_{11}u_{32}=e^{-4\pi i\theta}u_{32}u_{11},&& \left[u_{11},u_{33}\right]=0,&& u_{12}u_{13}=e^{-2\pi i\theta}u_{13}u_{12},\\
&u_{12}u_{21}=e^{4\pi i\theta}u_{21}u_{12},&& u_{12}u_{22}=e^{2\pi i\theta}u_{22}u_{12},&& \left[u_{12},u_{23}\right]=0,\\
&\left[u_{12},u_{31}\right]=0,&& u_{12}u_{32}=e^{-2\pi i\theta}u_{32}u_{12},&& u_{12}u_{33}=e^{-4\pi i\theta}u_{33}u_{12},\\
&\left[u_{13},u_{21}\right]=0,&& u_{13}u_{22}=e^{4\pi i\theta}u_{22}u_{13},&& u_{13}u_{23}=e^{2\pi i\theta}u_{23}u_{13},\\
&u_{13}u_{31}=e^{-4\pi i\theta}u_{31}u_{13},&& \left[u_{13},u_{32}\right]=0,&& u_{13}u_{33}=e^{-2\pi i\theta}u_{33}u_{13},\\
&u_{21}u_{22}=e^{-2\pi i\theta}u_{22}u_{21},&& u_{21}u_{23}=e^{2\pi i\theta}u_{23}u_{21},&& u_{21}u_{31}=e^{2\pi i\theta}u_{31}u_{21},\\
&\left[u_{21},u_{32}\right]=0,&& u_{21}u_{33}=e^{4\pi i\theta}u_{33}u_{21},&& u_{22}u_{23}=e^{-2\pi i\theta}u_{23}u_{22},\\
&u_{22}u_{31}=e^{4\pi i\theta}u_{31}u_{22},&& u_{22}u_{32}=e^{2\pi i\theta}u_{32}u_{22},&& \left[u_{22},u_{33}\right]=0,\\
&\left[u_{23},u_{31}\right]=0,&& u_{23}u_{32}=e^{4\pi i\theta}u_{32}u_{23},&& u_{23}u_{33}=e^{2\pi i\theta}u_{33}u_{23},\\
&u_{31}u_{32}=e^{-2\pi i\theta}u_{32}u_{31},&& u_{31}u_{33}=e^{2\pi i\theta}u_{33}u_{31},&& u_{32}u_{33}=e^{-2\pi i\theta}u_{33}u_{32}.
\end{alignat*}

The coalgebra structure restricted these elements is given by
\begin{align*}
\Delta(u_{ij})=\sum_{k=1}^3u_{ik}\otimes u_{kj},&\quad&\epsilon(u_{ij})=\delta_{ij},&\quad&S(u_{ij})=u_{ji}^*.
\end{align*}
Thus, $\sut$ is a compact quantum group.

While we consider only the $*$-algebra generated by elements to be our compact quantum group, Theorem 3.9 of \cite{SW1996} shows that 
the algebra generated by $\{u_{ij}\}$ is dense in $\sut$, and therefore, that $\sut$ can be completed to a compact matrix quantum group with 
the C$^*$-algebra structure.

%%%%%%%%%%%%%%%% Section 3 %%%%%%%%%%%%%%%%

\section{\bf actions of $\sut$}
\label{sec:action.S5}

In this section we present examples of actions of $\sut$. Classically, $SU(3)$ acts most naturally on $S^5$ (transitively) and 
on $SU(4)$ as a subgroup. In the first subsection, we review the action of $\sut$ on $\NCS5$ as a compact quantum group. This example was also 
presented in \cite{W2018}, but it remains relevant for the present article because the fixed-point subalgebra for this action is shown to be trivial. Thus, 
it gives one way of generalizing homogeneous spaces. On the other hand, the action of $\sut$ on $SU(4)_{\lambda}$ 
as a quantum group is a new example in Section \ref{subsec:homogeneous.space}. In this example, we derive the noncommutative 7-sphere 
$S^7_{\theta'}$ as the fixed-point algebra for this action. Therefore, it is a homogeneous space. In the noncommutative geometric framework, 
these two generalizations of homogeneous spaces seem to be different. 

\subsection{\bf action of $\sut$ on $\NCS5$}
\label{subsec:action.S5}
\noindent 
We construct the action of $\sut$ on $\NCS5$ \cite{W2018}. It is a generalization of the classical action of $SU(3)$ on $S^5$. 
The fixed-point subalgebra is trivial, which is analogous to the classical action of $SU(3)$ on $S^5$. However, the extent to 
which it remains an action in the deformation depends on the choice of the parameters of the deformations. 

An odd dimensional noncommutative (2n-1)-sphere $\NCS{2n-1}$ is a $\theta$-deformation of the algebra generated by the coordinate functions on the 
(2n-1)-sphere $S^{2n-1}$, which can be constructed from the action of $\T^n$. It is the $*$-algebra generated by $n$ normal elements $z_1,z_2,\ldots,z_n$ 
satisfying the commutation relations
\begin{align*}
z_jz_k=e^{2\pi i\lambda_{jk}}z_kz_j,\qquad z_jz_k^*=e^{-2\pi i\lambda_{jk}}z_k^*z_j,\qquad \sum_{k=1}^n z_kz_k^*=1
\end{align*}
where $\lambda_{jk} = - \lambda_{kj}$. These spheres are often referred to as Connes-Landi spheres.
%This is the same odd dimensional noncommutative sphere introduced by Natsume and Olsen in 
%\cite{Toeplitz Operators on Noncommutative Spheres and an Index Theorem}.

\begin{definition}
Let $H$ be a compact quantum group and $A$ an algebra. An algebra homomorphism
\begin{align*}
\rho:A\rightarrow H\otimes A
\end{align*}
is called a (left) action of the compact quantum group if 
\begin{enumerate}[i)]
\item $\left(\Delta\otimes id\right)\circ\rho=\left(id\otimes \rho\right)\circ\rho$ and 
\item $\left(\epsilon\otimes id\right)\circ\rho=id$.
\end{enumerate}
The set denoted by $A^H:=\{a\in A:\rho(a)=1\otimes a\}$ is called the fixed-point subalgebra for the action by $H$. 
The elements of $A^H$ are called invariant elements.
\end{definition}

We restate the following criterion, which gives the necessary and sufficient condition of which an action of the undeformed 
algebras extends to the deformed setting \cite{W2018}.

\begin{lemma}\label{lemma:action.condition}
Let $A$ be an algebra and $H$ a compact quantum group both equipped with $n$-torus $\T^n$ actions, $\alpha_t$ and $\beta_s$, respectively. Let 
$\lambda$ and $\theta=K\oplus(-K)$ be real antisymmetric matrices. 
Suppose $\rho:A\rightarrow H\otimes A$ is an action of the quantum group $H$ on $A$. 
Then, it extends to an action $\rho_{\theta,\lambda}:A_\lambda\rightarrow H_{\theta}\otimes A_{\lambda}$ 
if and only if for each $a=\underset{\vec n}{\sum}a_{\vec n}$ and $b=\underset{\vec m}{\sum}b_{\vec m}$ the following equation holds.
\begin{align}\label{eq:action.condition}
\begin{split}
&e^{\pi i\lambda(\vec n,\vec m)}\int_{\T^n\times \T^n}\rho(\alpha_t(a_{\vec n}))\rho(\alpha_s(b_{\vec m}))e^{-2\pi i(t\cdot\vec n + s\cdot\vec m)}\\
&=\sum_{(\vec{n'},\vec{p}),(\vec{m'},\vec{q})}
e^{2\pi i(\theta(\vec{n'},\vec{m'})+\lambda(\vec{p},\vec{q}))}\times\\
&\qquad\qquad\int_{\T^{2n}\times\T^{2n}}(\beta_{t'}\otimes\alpha_t)(\rho(a_{n}))(\beta_{s'}\otimes\alpha_s)(\rho(b_{n}))\times\\
&\qquad\qquad\qquad e^{-2\pi i(t'\cdot\vec{n'} + t\cdot\vec{p} + s'\cdot\vec{m'} + s\cdot\vec{q})}dt'dtds'ds
\end{split}
\end{align}
for each $\vec n$ and $\vec m$ in $\Z^n\times\Z^n$. 
Moreover, the fixed-point subalgebra $A_{\lambda}^{H_\theta}\subset A_{\lambda}$ is isomorphic to $A^H$.
\end{lemma}

\begin{proof}
Note that $a_{\vec n}=\int_{\T^n}\alpha_t(a)e^{-2\pi i t\cdot\vec n}$. For $\rho$ to be a homomorphism with respect to the new products, we need to have 
\[
\rho(a\times_{\lambda}b)=\rho(a)\times_{\theta\oplus\lambda}\rho(b).
\]
The left hand side gives
\begin{align*}
\rho(a\times_{\lambda}b)&=\rho\left(\sum_{\vec{n},\vec{m}}e^{\pi i\lambda(\vec{n,}\vec{m})}a_{\vec{n}}b_{\vec{m}}\right)\\
&=\sum_{\vec{n},\vec{m}}e^{\pi i\lambda(\vec{n,}\vec{m})}\rho(a_{\vec{n}})\rho(b_{\vec{m}})\\
&=\sum_{\vec{n},\vec{m}}e^{\pi i\lambda(\vec{n,}\vec{m})}\int_{\T^{n}\times\T^{n}}\rho(\alpha_t(a_{\vec n}))\rho(\alpha_s(b_{\vec m}))e^{-2\pi i(t\cdot\vec{n} 
+ s\cdot\vec{m})}dtds
\end{align*}
while the right hand side gives
\begin{align*}
\rho(a)\times_{\theta\oplus\lambda}\rho(b)&=\left(\sum_{\vec{n}}\rho(a_{\vec{n}})\right)\times_{\theta\oplus\lambda}\left(\sum_{\vec{m}}\rho(b_{\vec{m}})\right)\\
&=\sum_{\vec{n},\vec{m}}\rho(a_{\vec{n}})\times_{\theta\oplus\lambda}\rho(b_{\vec{m}})\\
&=\sum_{\vec{n},\vec{m}}\sum_{(\vec{n'},\vec{p}),(\vec{m'},\vec{q})}e^{2\pi i(\theta(\vec{n'},\vec{m'})+\lambda(\vec{p},\vec{q}))}\times\\
&\qquad\int_{\T^{2n}\times\T^{2n}}(\beta_{t'}\otimes\alpha_t)(\rho(a_{\vec n}))(\beta_{s'}\otimes\alpha_s)(\rho(b_{\vec m}))\times\\
&\qquad e^{-2\pi i(t'\cdot\vec{n'} + t\cdot\vec{p} + s'\cdot\vec{m'} + s\cdot\vec{q})}dt'dtds'ds~,
\end{align*}
which shows that $\rho$ is a homomorphism if and only if this condition \eqref{eq:action.condition} is satisfied. 

To show that $\rho$ is an action, note that the coproduct $\Delta$ and the action itself $\rho$ are unchanged. Therefore, the conditions i) and ii) of 
left action are automatically satisfied.

Since the action $\rho$ is unchanged, we see that the invariant elements remain unchanged and $A_{\lambda}^{H_\theta}$ is a subalgebra because 
$\rho$ is an algebra homomorphism. 
\end{proof}

Although the proof is simple, the previous lemma an important criterion of when an action extends to the deformed algebras. It suffices to show 
that the generators satisfy \eqref{eq:action.condition}.

For instance, using $S^5\subset\C^3$, the action of $SU(3)$ on $S^5$ can be given by the matrix multiplication $(z_1,z_2,z_3)^T\mapsto U(z_1,z_2,z_3)^T$, 
$(z_1,z_2,z_3)^T\in S^5\subset\C^3$ and $U\in SU(3)$. The dual version of this action is the action given by $z_j\mapsto\sum_{k=1}^3u_{jk}\otimes z_k$.

Let $H_\theta=\sut$ and $A_{\lambda}=\NCS5$. Then, an action of the compact quantum group $H_\theta$ on $A_{\lambda}$ 
can be given by the following proposition.

\begin{proposition}\label{proposition:cotransitive}
Let $\lambda=(\lambda_{ij})$ be an real antisymmetric matrix whose entries satisfy $  \lambda_{12} = - \lambda_{13} =  \lambda_{23}=\theta$. 
Then for this particular choice of the values of the parameters, % of $A_{\lambda}$ and $H_\theta$, 
the linear map $\delta:A_\lambda\rightarrow H_\theta\otimes A_\lambda$ defined by $\delta(z_j)=\sum_{k=1}^3u_{jk}\otimes z_k$ 
is a left action of the compact quantum group $H_\theta$ on the algebra $A_{\lambda}$. 
This action is ergodic in the sense that $A_{\lambda}^{H_\theta}\cong\C$.
\end{proposition}

\begin{proof}
We show that the equation \eqref{eq:action.condition} holds for those choices of values of $\lambda_{jk}$. For instance, 
\begin{align*}
\delta(z_1)&=u_{11}\otimes z_1+u_{12}\otimes z_2+u_{13}\otimes z_3\\
\delta(z_2)&=u_{21}\otimes z_1+u_{22}\otimes z_2+u_{23}\otimes z_3.
\end{align*}
Now,
\begin{align*}
\delta(z_1\times_{\lambda}z_2) = e^{\pi i\lambda_{12}} \delta(z_1z_2)= 
\delta(z_1)\times_{\theta\oplus\lambda}\delta(z_2)
\end{align*}
while 
\begin{align}
\begin{split}
\label{eq:coactionS5}
&\delta(z_1)\times_{\theta\oplus\lambda}\delta(z_2)\\
&=e^{\pi i\theta}u_{11}u_{21}\otimes z_1z_1+e^{\pi i\lambda_{12}}u_{11}u_{22}\otimes z_1z_2+e^{\pi i(2\theta + \lambda_{13})}u_{11}u_{23}\otimes z_1z_3\\
&+e^{\pi i(2\theta - \lambda_{12})}u_{12}u_{21}\otimes z_2z_1 + e^{\pi i\theta}u_{12}u_{22}\otimes z_2z_2+e^{\pi i\lambda_{23}}u_{12}u_{23}\otimes z_2z_3\\
&+e^{ - \pi i\lambda_{13}}u_{13}u_{21}\otimes z_3z_1+e^{\pi i(2\theta - \lambda_{23})}u_{13}u_{22}\otimes z_3z_2+e^{\pi i\theta}u_{13}u_{23}\otimes z_3z_3.
\end{split}
\end{align}
From such relation \eqref{eq:coactionS5}, the necessary condition for the values of $\lambda=\left(\lambda_{jk}\right)$ are already restricted to 
$\lambda_{12}= - \lambda_{13}=\lambda_{23}=\theta$. The equation \eqref{eq:action.condition} shows that it is enough to prove such relations for the 
isotypic components of the algebra. The commutation relations for other generators can be ccomputed similarly. Thus, $\delta$ is an action whose 
fixed-point subalgebra is $\C$ by Lemma \ref{lemma:action.condition}.
\end{proof}

\subsection{\bf $\theta$-deformation of 7-sphere as a homogeneous space}
\label{subsec:homogeneous.space}

\noindent
In this subsection, we give the noncommutative 7-sphere $S^7_{\theta'}$ an interpretation of a homogeneous space. We achieve this notion by interpreting 
$S^7_{\theta'}$ as the fixed-point subalgebra for the action by $SU(3)_\theta$ on $SU(4)_{\lambda}$. More generally, we can consider homogeneous spaces 
$M=G/K$ where $G$ is a compact Lie group and $K$ is a closed subgroup. Using the isomorphism $C(G)^K\rightarrow C(M)$ of algebras between the algebra $C(M)$ of functions on $M$ 
the $K$-invariant functions $C(G)$, the algebra of functions on the homogeneous space is defined to be $C(G)^K$. This notion can be generalized to the noncommutative setting.
%It can be shown that this is an isomorphism of C$^*$-algebras of continuous functions. %Thus, we make the following definition.

%\begin{definition}
%Let $H$ be a compact quantum group and $K$ be a subalgebra such that 
%\end{definition}

A 3-parameter deformation $SU(4)_{\lambda}$ of $SU(4)$ using the action of its maximal torus $\T^3$ can be computed using the same method for 
$SU(3)_\theta$ in Section \ref{subsec:SU(3)}. Now we can determine exactly when $SU(3)_\theta$ acts on $SU(4)_{\lambda}$ using Lemma 
\ref{lemma:action.condition} and explicitly determine the invariant 
elements. The proof is essentially identical to Proposition \ref{proposition:cotransitive} with a 
minor modification. As a subgroup, $SU(3)\subset SU(4)$ acts on $SU(4)$ by the left matrix multiplication on the left upper 3$\times$3 block. The following 
theorem is the quantum counterpart of this action.

\begin{proposition}
\label{proposition:homogeneous.space}
Let $u_{ij}\in SU(3)_\theta$ and $v_{kl}\in SU(4)_{\lambda}$. The map defined by 
\begin{align*}
\rho(v_{kl})&=\sum_{\alpha=1}^3u_{k\alpha}\otimes v_{\alpha l},~~~~k,l=1,2,3\\
\rho(v_{4l})&=1\otimes v_{4l},~~~~k,l=1,2,3,4
\end{align*}
is an action $\rho:SU(4)_{\lambda}\rightarrow SU(3)_{\theta}\otimes SU(4)_{\lambda}$ %(algebra homomorphism) 
if and only if $\theta=\lambda_{12}=-\lambda_{13}=\lambda_{23}$. In this case, the algebra $B=\left(SU(4)_{\lambda}\right)^{SU(3)_\theta}$ of invariant elements is generated by 
$$
\left\{1\otimes v_{4l}:l=1,2,3,4\right\}.
$$
\end{proposition}
Set $x_l=1\otimes v_{4l}$. Then the algebra $B$ generated by $x_l$ is isomorphic to the noncommutative 7-sphere $S^7_{\theta'}$ as in \cite{C1994,CL2001,LvS2005} with 
$$
\theta'=
\begin{pmatrix}
0       &  -\theta  & \theta &0\\
\theta    & 0 &  -\theta  & 0\\
-\theta   & \theta  & 0 & 0\\
0  & 0 & 0 & 0\\
\end{pmatrix}.
$$ 
Exact relations can be computed using the commutation relations of $SU(4)_{\lambda}$ and $\theta=\lambda_{12}=-\lambda_{13}=\lambda_{23}$ as in the above theorem:

\begin{alignat}{11}
\nonumber
&x_1x_2=e^{-2\pi i\theta}x_2x_1,&&\quad&& x_1x_3=e^{2\pi i\theta}x_3x_1,&&\quad&& x_1x_4=x_4x_1,&\\
\nonumber
&x_2x_3=e^{-2\pi i\theta}x_3x_2,&&\quad&& x_2x_4=x_4x_2,&& \quad&& x_3x_4=x_4x_3,&\\
\nonumber 
&x_1x_1^*+x_2x_2^*+ x_3x_3^*+x_4x_4^*=1.%&&&&&&&&&
\end{alignat}

Although the $\theta$-deformations of compact manifolds had been considered as homogeneous spaces, our result yields a general construction. It shows that it is 
enough to compute the dependence of the deformation parameters according to \eqref{eq:action.condition}. 
For instance, Varilly in \cite{V2001} concluded that the some of the generators of the odd dimensional noncommutative spheres commute with everything else, 
which is consistent with our result since $x_4$ in the centre. On the other hand, the noncommutative 3-sphere $S^3_{\theta'}$ is not a homogeneous space unless $\theta=0$ 
while $S^5_{\lambda}$ is only a homogeneous space in the sense that the fixed-point algebra for the $SU(3)_\theta$ action is trivial in the present work. Indeed, the natural 
left action by $SU(2)$ on $SU(3)$ given by 
\begin{align}
\rho\left(u_{ij}\right)=\begin{cases}
\alpha u_{1j}+\beta u_{2j} & i=1\\
-\bar{\beta}u_{1j}+\alpha u_{2j} & i=2\\
u_{3j} & i=3
\end{cases}
\end{align}
does not extend on $SU(3)_\theta$. Simply, $\rho(u_{11}u_{12})\neq\rho(e^{-2\pi i\theta}u_{12}u_{11})$.
Varilly's method differs from ours. Rather than computing the invariant elements, Varilly endows $C(G/K)$ with a new product consistent in a way 
that it is embedded in $C(G)_\theta$. We took a more direct approach to endow $C(G)$ with the natural action by $C(K)$ and studied the extent to which 
the original action is an action in the $\theta$-deformation context. 
In that sense, it does not make sense to consider $S^5_{\lambda}$ as an embeddable homogeneous space in $SU(3)_\theta$ for $SU(2)$.

The ubiquity of the $\theta$-deformation in noncommutative geometry is already familiar.
While the present work shows merely an example of symmetry in $\theta$-deformation, 
it explains the pervasiveness of symmetry in the $\theta$-deformation. There are already numerous research in mathematical physics using a 
toric noncommutative manifold as models, whether they are realized as $\theta$-deformations or not. We believe that our work has potential 
to reveal symmetries of toric noncommutative manifolds, and therefore, it has potential to be used in numerous areas of quantum physics.

Mathematically, it would be an interesting problem to classify up to Morita equivalence or isomorphism these objects. In fact, very little is known about the 
Morita equivalence of the $\theta$-deformations other than the case of the noncommutative 2-torus. 

\section*{Acknowledgment}
I would like to gratefully thank the organizers for organizing XXII Coloquio Latinoamericano de \'Algebra, 2017 at Pontificial Universidad cat\'olica del Ecuador 
where I was able to advance my work. I would like to thank Sylvie Paycha for some fruitful discussions at this conference.


\begin{thebibliography}{9}
\bibitem{AW20171} J. Arnlind and M. Wilson. Riemannian curvature of the noncommutative 3-sphere. J. of Noncommut. Geom. 11, 2 (2017) 507-536.
\bibitem{AW20172} J. Arnlind and M. Wilson. On the Gauss-Bonnet-Chern type theorem for the noncommutative 4-sphere. J. of Geom. and Phys., 111 (2017) 126-141.
\bibitem{PAFB2002}  P. Aschieri and F. Bonechi. On the noncommutative geometry of twisted spheres. Letters in Math. Phys. 59 (2002) 133-156.
\bibitem{BH2004} T. Brzezi\'nski and P. Hajac. The Chern-Galois character. C.R. Acad. Sci. Paris Ser. I 338 (2004) 113-116.
\bibitem{BDZ2004} T. Brzezi\'nski, L. D\c{a}browski and B. Zielinski. Hopf fibration and monopole connection over the contact quantum spheres. J. Geom. Phys. 51 (2004) 71-81.
\bibitem{C1994} A. Connes. Noncommutative Geometry. Academic Press, San Diego, 1994.
\bibitem{CD2002} A. Connes and M. Dubois-Violette, Noncommutative finite-dimensional manifolds. I. Spherical
manifolds and related examples, Commun. Math. Phys. 230 (2002) 539-579.
\bibitem{CL2001} A. Connes and G. Landi. Noncommutative manifolds: the instanton algebra and isospectral
deformations, Commun. Math. Phys. 221 (2001) 141-159.
\bibitem{H1996}  P. Hajac. Strong Connections on Quantum Principal Bundles. Comm. Math. Phys. 182, 579-617 (1996).
\bibitem{BJPS2002} B. Jurco, P. Schupp and J. Wess, Noncommutative line bundle and Morita equivalence. Lett. Math.
Phys. 61 (2002), no. 3, 171-186.
\bibitem{K1995} C. Kassel. Quantum Groups. Springer-Verlag, New York (1995).
\bibitem{LvS2005}  G. Landi and W. van Suijlekom. Principal fibrations from noncommutative spheres.
Commun. Math. Phys. 260 (2005) 203-225.
\bibitem{LPR2006} G. Landi, C. Pagani, and C. Reina. A Hopf bundle over a quantum four-sphere
from the symplectic group. Commun. Math. Phys. 263 (2006) 65-88.
\bibitem{M1995} S. Majid, Foundations of quantum group theory, Cambridge, UK: Univ. Pr. (1995).
\bibitem{MR1990} M.Rieffel. Non-commutative tori - A case study of non-commutative differentiable
manifolds. Contemp. Math. 105 (1990) 191-212.
\bibitem{MR1993.5} M. Rieffel. Quantum compact groups associated with toral subgroups. Contemporary Mathematics. 145, 1993.
\bibitem{MR1993} M. Rieffel. Deformation quantization for actions of $\R^d$. Mem. Amer. Math. Soc. 106 (1993) 506.
\bibitem{Si2001} A. Sitarz. Rieffel's deformation quantization and isospectral deformations. International J. of Theoretical Phys. (2001) 40- 1693.
\bibitem{V2001} J. V\'arilly. Quantum symmetry groups of noncommutative spheres. Comm. in Math. Phys. 221 (2001), Issue 3, 511-523
\bibitem{SW1996} S. Wang. Deformations of Compact Quantum Groups via Rieffel's Quantization. Comm. Math. Phys. 178 (1996) 747-764.
\bibitem{WR2017} M. Wilson and C. Rodriguez. Pseudo-differential calculus on deformation quantization of compact Lie groups. In preparation.
\bibitem{W2018} M. Wilson. Quantum symmetries of the deformation quantization of compact Lie groups. Submitted to Letters in Mathematical Physics, 2018.
%\bibitem{W2017.0} M. Wilson. Deformation quantization of principal bundles by the action of $\T^n$. In preparation (2017).
%\bibitem{W2017.1} M. Wilson. Associated bundles and strong connection for the deformation quantization of Hopf-Galois extensions. In preparation (2017).
\bibitem{W1987} S. Woronowicz. Compact matrix pseudogroups. Comm. Math. Phys. 111 (1987), 613-665.
\bibitem{NY1990} N.Yu. Reshetikhin, Multiparameter Quantum Groups and Twisted Quasitriangular Hopf Algebras.
Letters in Mathematical Physics 20 (1990) 331-335.
\end{thebibliography}
\end{document}